\documentclass[12pt]{amsart}
\newtheorem{thm}{Theorem}[section]
\newtheorem{prop}{Proposition}[section]
\newtheorem{lemma}{Lemma}[section]

\theoremstyle{definition}
\newtheorem{defn}{Definition}[section]
\newtheorem{rmk}{Remark}[section]
\newcommand{\intprod}{\;\rule{5pt}{.3pt}\rule{.3pt}{7pt}\;}
\newcommand{\scriptintprod}{\;\rule{3pt}{.3pt}\rule{.3pt}{5pt}\;}
\newcommand{\mapsisoto}{
   \makebox[0pt][l]{\raisebox{5pt}{\hspace*{3.8pt}$\simeq$}}\longrightarrow}
\DeclareSymbolFont{script}{U}{eus}{m}{n}
\DeclareMathSymbol{\Wedge}{0}{script}{"5E}
\begin{document}
\title[Sub-Riemannian contact manifolds]
{A canonical connection on sub-Riemannian contact manifolds}
\author[Michael Eastwood]{Michael Eastwood}
\address{\hskip-\parindent
School of Mathematical Sciences\\
University of Adelaide,\newline
SA 5005, Australia}
\email{meastwoo@member.ams.org}
\author[Katharina Neusser]{Katharina Neusser}
\address{\hskip-\parindent
Eduard \v Cech Institute $\&$ Mathematical Institute\\ 
Charles University,\newline
Prague, Czech Republic}
\email{katharina.neusser@gmail.com}
\begin{abstract}
We construct a canonically defined affine connection in sub-Riemannian contact
geometry. Our method mimics that of the Levi-Civita connection in Riemannian
geometry. We compare it with the 
Tanaka-Webster connection in the three-dimensional case. 
\end{abstract}
\renewcommand{\subjclassname}{\textup{2010} Mathematics Subject Classification}
\thanks{The second author was supported by the Eduard \v Cech Institute.}
\mbox{ }

\vspace{50pt}
\maketitle

\section{Introduction}\label{intro}
Let $M$ be a contact manifold with contact distribution~$H$. Necessarily, the
dimension of $M$ is odd. The $3$-dimensional case is special and we shall often
concentrate our discussion on this case. For example, as observed
in~\cite{fgv}, the notion of a sub-Riemannian structure in this case coincides
with Webster's notion of a pseudo-Hermitian structure~\cite{w}. {From} this
point of view, there is the well-known Tanaka-Webster connection~\cite{t,w}, a
canonically defined affine connection on pseudo-Hermitian manifolds in all
dimensions but, in particular, on sub-Riemannian manifolds in dimension~$3$. We
shall discuss this connection in detail~\S\ref{others} and compare it with what
we constructed earlier in~\S\ref{ours}. There is yet another natural connection
for sub-Riemannian contact structures due to Morimoto~\cite{m}. In contrast to
the canonical connection we construct in~\S\ref{ours}, this one requires
`constant symbol,' which is however a vacuous condition in dimension~$3$.

We admit right away that our aim here is not to discuss `the equivalence
problem' for sub-Riemannian contact structures in the sense of Cartan nor study
`Jacobi curves' in sub-Riemannian geometry the sense of Agrachev and
Zelenko~\cite{az}. Instead, our more modest aim is to discuss and construct
connections and partial connections (where, in the first instance, one
differentiates only the $H$-directions) as an invariant calculus in
sub-Riemannian contact geometry, closely mimicking the construction of the
Levi-Civita connection in the Riemannian setting.

\section{Generalities on contact manifolds}
Let $M$ be a smooth manifold with tangent bundle $TM\to M$ and suppose
$H\subset TM$ is a codimension~$1$ smooth subbundle. Equivalently, and this is
our preferred point of view, we are given a smooth line subbundle 
$L\subset\Wedge^1$, where $\Wedge^1=T^*M$ is the bundle of $1$-forms on~$M$. 
Thus, we have dual exact sequences
$$0\to H \to TM \to L^*\to 0$$
and
\begin{equation}\label{L}0\to L\to\Wedge^1\to\Wedge_H^1\to 0.\end{equation}
\subsection{The Levi form}
The sequence (\ref{L}) induces a short exact sequence
$$0\to \Wedge_H^1\otimes L\to\Wedge^2\to\Wedge_H^2\to 0,$$
where $\Wedge^2_H=\Wedge^2(\Wedge_H^1)$ and we may now consider the diagram
$$\begin{array}{ccccccccc}
0&\to&L&\to&\Wedge^1&\to&\Wedge_H^1&\to&0\\
&&&&{\scriptstyle d}\!\downarrow\phantom{d}\\
0&\to&\!\!\Wedge_H^1\otimes L\!\!&\to&\Wedge^2&\to&\Wedge_H^2&\to&0,
\end{array}$$
where $d:\Wedge^1\to\Wedge^2$ is the exterior derivative. {From} the Leibniz 
rule, it follows that the composition 
$$L\to\Wedge^1\xrightarrow{\,d\,}\Wedge^2\to\Wedge_H^2$$
is a homomorphism of vector bundles. 

\begin{defn} This composition ${\mathcal{L}}\in\Gamma(\Wedge_H^2\otimes L^*)$
is called the {\em Levi form\/} of $H$. (It is the obstruction to the
integrability of~$H$.)
\end{defn}

\begin{defn} If ${\mathcal{L}}$ is non-degenerate, then $(M,H)$ is said to be 
a {\em contact manifold}. (It follows that $M$ is odd-dimensional.) 
\end{defn}
If $\theta\in\Gamma(L)\subset\Gamma(\Wedge^1)$ is nowhere vanishing, 
non-degeneracy of the Levi form is equivalent to
$$\theta\wedge(d\theta)^n\equiv
\theta\wedge \underbrace{d\theta\wedge\cdots\wedge d\theta}_n\not=0,$$
where $2n+1$ is the dimension of~$M$.

\begin{defn} On a smooth manifold of dimension $2n+1$, a {\em contact form\/} 
is a smooth $1$-form $\theta$ such that $\theta\wedge(d\theta)^n\not=0$.
\end{defn}

\begin{rmk} Some authors define a contact manifold as a smooth manifold 
equipped with a contact form. In this article, however, a contact form is an 
extra choice.
\end{rmk}

\subsection{Partial connections}
On a contact manifold it is natural to consider differentiation in the contact
directions, i.e.\ along~$H$. According to (\ref{L}), this is equivalent to 
considering the composition
$$\Wedge^0\xrightarrow{\,d\,}\Wedge^1\to\Wedge_H^1,$$
which we denote by $d_H:\Wedge^0\to\Wedge_H^1$. A consequence of the 
contact condition is that $H$ is {\em bracket generating\/} and a consequence 
of this is that the kernel of $d_H$ consists of locally constant functions. 
\begin{rmk}
In~\cite{r} it is shown that $d_H:\Wedge^0\to\Wedge_H^1$, is the first operator
in an invariantly defined locally exact complex, known as the {\em Rumin\/}
complex. It is an effective replacement for the de~Rham complex (see
also~\cite{begn}).
\end{rmk}
\begin{defn} Suppose $M$ is a smooth contact manifold and $V$ is a smooth 
vector bundle on~$M$. A {\em partial connection} on $V$ is a differential 
operator
$$\nabla_H:V\to\Wedge_H^1\otimes V\quad\mbox{s.t.}\enskip
\nabla_H(f\sigma)=f\nabla_H\sigma + d_H\!f\otimes\sigma,$$
for all $f\in\Gamma(\Wedge^0)$ and $\sigma\in\Gamma(E)$.   
\end{defn}
\begin{rmk}A partial connection determines a differential operator
$$\nabla_H:\Wedge^1\otimes V\to\Wedge_H^2\otimes V$$
characterised by 
$$\nabla_H(\omega\otimes\sigma)=
d_H\omega\otimes\sigma-\omega_H\wedge\nabla_H\sigma,$$
where $d_H:\Wedge^1\to\Wedge_H^2$ is the composition 
$\Wedge^1\xrightarrow{\,d\,}\Wedge^2\to\Wedge_H^2$ and 
$\omega_H$ the image of $\omega$ under the projection
$\Wedge^1\to\Wedge_H^1$. 
\end{rmk}
In fact, a partial connection on any bundle on any contact manifold may be
promoted to a full connection as follows. The Levi form 
${\mathcal{L}}:L\to\Wedge_H^2$ is non-degenerate and so has 
an `inverse' ${\mathcal{L}}^{-1}:\Wedge_H^2\to L$ (defined on all of 
$\Wedge_H^2$ and an inverse on the range of~${\mathcal{L}}$). 
\begin{prop}\label{promotion}
A partial connection $\nabla_H$ on $V$ uniquely determines a connection 
$\nabla$ on $V$ with the following two properties.
\begin{itemize}
\item the composition 
$V\xrightarrow{\,\nabla\,}\Wedge^1\otimes V\to\Wedge_H^1\otimes V$ agrees 
with~$\nabla_H$,
\item the composition
$V\xrightarrow{\,\nabla\,}\Wedge^1\otimes V\xrightarrow{\,\nabla_H}
\Wedge_H^2\otimes V\xrightarrow{\,{\mathcal{L}}^{-1}\otimes{\mathrm{Id}}\,}
L\otimes V$ vanishes.
\end{itemize}
\end{prop}
\begin{proof} See~\cite[Proposition~3.5]{eg}.
\end{proof}
\subsection{Partial torsion}
Suppose $\nabla_H:\Wedge^1\to\Wedge_H^1\otimes\Wedge^1$ is a partial connection
on the cotangent bundle. Then we have two linear differential operators between
the same bundles, namely
\begin{equation}\label{two_operators}\begin{array}{rl}
\bullet&d_H:\Wedge^1\to\Wedge_H^2,\\
\bullet&\mbox{the composition } 
\Wedge^1\xrightarrow{\,\nabla_H\,}\Wedge_H^1\otimes\Wedge^1\to
\Wedge_H^1\otimes\Wedge_H^1\to\Wedge_H^2,
\end{array}\end{equation}
both of which satisfy a Leibniz rule, e.g.
$$d_H(f\omega)=fd_H\omega+d_H\!f\wedge\omega_H.$$ 

\begin{defn}The difference between the two differential operators in
(\ref{two_operators}) is called the {\em partial torsion\/} $\tau_H$
of~$\nabla_H$. The Leibniz rule implies that it is a homomorphism of bundles, 
equivalently $\tau_H\in\Gamma(\Wedge_H^2\otimes TM)$.
\end{defn}
\begin{rmk}
If $\nabla_H$ preserves $L$, then the projection of $\tau_H$ to 
$\Wedge_H^2\otimes L^*$ is the Levi form ${\mathcal{L}}$ of~$H$.
\end{rmk}
\begin{lemma}\label{remove_partial_torsion}
On any smooth contact manifold, there are partial connections on $\Wedge^1$ 
with vanishing partial torsion. 
\end{lemma}
\begin{proof} By partition of unity one can choose a partial connection 
$\nabla_H$ on $\Wedge^1$ and then the general such partial connection is of 
the form $\nabla_H-\Gamma_H$ for an arbitrary homomorphism 
$\Gamma_H:\Wedge^1\to\Wedge_H^1\otimes\Wedge^1$. The partial torsion $\tau_H$ 
of $\nabla_H$ is then modified by the composition
$\Wedge^1\xrightarrow{\,\Gamma_H\,}\Wedge_H^1\otimes\Wedge^1\to\Wedge_H^2$
and so we may always adopt such a modification to ensure that $\nabla_H$ is
partially torsion-free, as required.\end{proof}
\begin{rmk}The remaining freedom in choosing a partially torsion-free 
connection on $\Wedge^1$ is $\nabla_H\mapsto\nabla_H-\Gamma_H$, where 
$$\Gamma_H:\Wedge^1\to\ker:\Wedge_H^1\otimes\Wedge^1\to\Wedge_H^2$$
is arbitrary.\end{rmk}

Now let us suppose that $\theta\in\Gamma(L)$ is nowhere vanishing. Such a
contact form may be used to effect a number of normalisations. Firstly, the
line bundle $L$ is trivialised. Secondly, a vector field $T$, called the {\em
Reeb\/} field, may be uniquely characterised by 
\begin{equation}\label{reeb}T\intprod\theta=1\qquad T\intprod d\theta=0.
\end{equation}
Consequently, the short exact sequence (\ref{L}) splits and we may write
\begin{equation}\label{splitting}\Wedge^1=\begin{array}{c}\Wedge_H^1\\ \oplus\\
\Wedge^0\end{array}\enskip\mbox{by means of}\quad\omega\mapsto
\left[\begin{array}{c}\omega_H\\ T\intprod\omega\end{array}\right].
\end{equation}
Equivalently, we may identify $\Wedge_H^1$ as a subbundle of $\Wedge^1$ by 
means of
$$\Wedge_H^1=
\ker:\Wedge^1\xrightarrow{\,T\scriptintprod\underbar{\enskip}\,}\Wedge^0$$
and $\Wedge_H^2$ as a subbundle of $\Wedge^2$ by means of
$$\Wedge_H^2=
\ker:\Wedge^2\xrightarrow{\,T\scriptintprod\underbar{\enskip}\,}\Wedge^1.$$
In particular, the $2$-form $d\theta$ may be viewed as a section of
$\Wedge_H^2$. It coincides with the image of $\theta$ under the Levi form
${\mathcal{L}}:L\to\Wedge_H^2$. Thus, in the presence of a contact form
$\theta$, we obtain a non-degenerate $2$-form 
$\Omega\equiv d\theta\in\Gamma(\Wedge_H^2)$ on the contact distribution~$H$. In
any case, we may use the splitting (\ref{splitting}) to insist that a partial
connection on $\Wedge^1$ have the form
\begin{equation}\label{we_insist}
\Wedge^1=\begin{array}{c}\Wedge_H^1\\ \oplus\\ 
\Wedge^0\end{array}\ni
\left[\begin{array}{c}\sigma\\ \rho\end{array}\right]
\stackrel{\nabla_H}{\longmapsto}
\left[\begin{array}{c}D_H\sigma+\Omega\rho\\ 
d_H\rho\end{array}\right]\in\begin{array}{c}\Wedge_H^1\otimes\Wedge_H^1\\ 
\oplus\\ \Wedge_H^1\end{array},\end{equation} 
where $D_H:\Wedge_H^1\to\Wedge_H^1\otimes\Wedge_H^1$ is a partial connection
on~$\Wedge_H^1$. The form of $\nabla_H$ ensures that its partial torsion lies
in~$\Wedge_H^2\otimes H$. Therefore, if we argue as in the proof of
Lemma~\ref{remove_partial_torsion} to remove the remaining partial torsion,
then we have shown the following: 
\begin{prop}\label{contact_to_the_max}If a contact form $\theta$ is
used to split the $1$-forms as
$$\Wedge^1=\Wedge_H^1\oplus\Wedge^0,$$
then we may find partial connections on $\Wedge^1$ of the form 
\eqref{we_insist} and free from partial torsion. The remaining freedom in 
choosing such connections is
$$\textstyle D_H\mapsto D_H-\Gamma_H\quad\mbox{for}\quad
\Gamma_H:\Wedge_H^1\to\bigodot^2\!\Wedge_H^1$$
an arbitrary homomorphism.\end{prop}
\begin{rmk}The connection dual to (\ref{we_insist}) has the form
$$TM=\begin{array}{c}\Wedge^0\\ \oplus\\ H\end{array}\ni
\left[\begin{array}{c}\lambda\\ X\end{array}\right]
\stackrel{\nabla_H}{\longmapsto}
\left[\begin{array}{c}d_H\lambda+X\intprod\Omega\\ D_HX\end{array}\right].$$
In particular, it follows that
$$\begin{array}{c}\Wedge^0\\ \oplus\\ H\end{array}\ni
\left[\begin{array}{c}1\\ 0\end{array}\right]
\stackrel{\nabla_H}{\longmapsto}
\left[\begin{array}{c}0\\ 0\end{array}\right]\quad\mbox{and}\quad
\begin{array}{c}\Wedge_H^1\\ \oplus\\ 
\Wedge^0\end{array}\ni
\left[\begin{array}{c}0\\ 1\end{array}\right]
\stackrel{\nabla_H}{\longmapsto}
\left[\begin{array}{c}\Omega\\ 0
\end{array}\right].$$
Evidently, these two conditions are sufficient to guarantee that a partial
connection on $\Wedge^1$ have the form (\ref{we_insist}) and
Proposition~\ref{contact_to_the_max} may be invariantly reformulated as
follows.
\end{rmk}
\begin{thm}\label{contact_theorem}
If $\theta$ is a contact form with associated Reeb field~$T$, then
we may find partial connections on the (co-)tangent bundle such that
\begin{itemize}
\item $\nabla_HT=0$,
\item $\nabla_H\theta=(d\theta)_H$,
\item $\nabla_H$ is free from partial torsion,
\end{itemize}
where $(d\theta)_H$ is the image of $d\theta$ under the composition
$$\Wedge^2\hookrightarrow\Wedge^1\otimes\Wedge^1\to\Wedge_H^1\otimes\Wedge^1.$$
The freedom in choosing such a partial connection lies in 
$\Gamma(\bigodot^2\!\Wedge_H^1\otimes H)$.
\end{thm}
\begin{rmk} Theorem~\ref{contact_theorem} bears a striking similarity to the
usual story for connections on the (co-)tangent bundle in which torsion-free
connections are free up to $\Gamma(\bigodot^2\!\Wedge^1\otimes TM)$. This
appealing feature is one of our reasons for advocating the construction in
this article.
\end{rmk}

\section{Sub-Riemannian contact geometry}
A sub-Riemannian contact structure on a smooth manifold $M$ is a contact
distribution $H\subset TM$ equipped with a positive-definite symmetric form
$g:\bigodot^2\!H\to{\mathbb{R}}$. We do not suppose any particular
compatibility between $g$ and the Levi form. 
For any chosen contact form
$\theta\in\Gamma(L)\subset\Gamma(\Wedge^1)$, however, we can chose a local 
co-frame for $H$ in which $g\in\Gamma(\bigodot^2\!\Wedge_H^1)$ and 
$\Omega=d\theta\in\Gamma(\Wedge_H^2)$ are simultaneously represented by the 
matrices
\begin{equation}\label{lovely_co-frame}
\left[\begin{array}{ccccc}
1&0&\cdots&0&0\\
0&1&\cdots&0&0\\
\vdots&\vdots&\mbox{\Large${\mathrm{Id}}$}&\vdots&\vdots\\
0&0&\cdots&1&0\\
0&0&\cdots&0&1
\end{array}\right]\quad\mbox{and}\quad
\left[\begin{array}{ccccc}
0&\lambda_1&\cdots&0&0\\
\!-\lambda_1\!&0&\cdots&0&0\\
\vdots&\vdots&\mbox{\Large$\!\ddots\!$}&\vdots&\vdots\\
0&0&\cdots&0&\lambda_n\\
0&0&\cdots&\!-\lambda_n\!&0
\end{array}\right],\end{equation}
respectively.
\begin{prop}\label{normalise_theta}
Locally, we can always choose a contact form $\theta$ so that
$$\|\Omega\|^2=2n,\quad\mbox{equivalently}\quad
\lambda_1{}^2+\cdots+\lambda_n{}^2=n.$$
With this normalisation $\theta$ is then determined up to sign.
\end{prop}
\begin{proof} Replacing $\theta$ by $\hat\theta=\lambda\theta$, for $\lambda$ a
nowhere vanishing smooth function, gives 
$d\hat\theta=\lambda d\theta+d\lambda\wedge\theta$ but, since $\theta$ 
vanishes on~$H$, as far as $\Wedge_H^2$ is concerned we find that 
$\hat\Omega=\lambda\Omega$. The stated normalisation and freedom are clear.
\end{proof}
\begin{rmk} If $M$ is three-dimensional, this normalisation asserts that
$\lambda_1=\pm1$ and a choice of sign corresponds to a choice of orientation
for~$H$. In this case, we may define an endomorphism $J:H\to H$ by 
$$g(JX,Y)=\Omega(X,Y),\quad\forall\;X,Y\in H$$
and our normalisation asserts that $J^2=-{\mathrm{Id}}$. Thus, we have obtained
a CR structure. Conversely, every {\em pseudo-Hermitian\/} structure in the
sense of Webster~\cite{w} arises in this way. More precisely, we have shown 
the following (as already noted in~\cite{fgv}).
\end{rmk}
\begin{prop}\label{3D} In three dimensions, an oriented sub-Riemannian
\mbox{contact}
structure is equivalent to a CR structure with a choice of contact form.
\end{prop}
\begin{rmk}
Without the contact form, a three-dimensional CR structure coincides with an
oriented `sub-conformal' contact structure (as in~\cite{fgv}).
\end{rmk}

\section{Construction of the partial connection}\label{ours}
The existence and uniqueness of the Levi-Civita connection in Riemannian
geometry is based on the algebraic fact that, for any finite-dimensional vector
space~$V$, the composition
\begin{equation}\label{algebra}
\textstyle\bigodot^2\!V\otimes V\hookrightarrow V\otimes V\otimes V
\xrightarrow{\,
\underbar{\enskip}\otimes\underbar{\enskip}\odot\underbar{\enskip}\,}
V\otimes\bigodot^2\!V\end{equation}
is an isomorphism. The same algebra underlies the following construction.
\begin{thm}\label{mainthm}
On any sub-Riemannian contact manifold,  there is a
unique partial connection $\nabla_H:\Wedge^1\to\Wedge_H^1\otimes\Wedge^1$ with
the following properties.
\begin{itemize}
\item $\nabla_HT=0$,
\item $\nabla_H\theta=(d\theta)_H$,
\item $\nabla_H$ is free from partial torsion,
\item $\nabla_Hg=0$,
\end{itemize}
where $\theta$ is any local contact form normalised as in 
Proposition~\ref{normalise_theta} and $T$ is its associated Reeb field.
\end{thm}
\begin{proof}
The only freedom in $\theta$ is to change its sign. Evidently, such a change
respects the characterising properties of~$\nabla_H$ so it suffices to work 
locally, choose~$\theta$, and employ Theorem~\ref{contact_theorem} to find a 
global connection with the first three of our required properties and with 
remaining freedom
$$\nabla_H\mapsto\hat\nabla_H=\nabla_H-\Gamma_H,$$
for $\Gamma_H\in\Gamma(\bigodot^2\!\Wedge_H^1\otimes H)$. If we use the
sub-Riemannian metric $g$ to identify $H$ with its dual~$\Wedge_H^1$, and
write $\sigma:\bigodot^2\!\Wedge_H^1\otimes\Wedge_H^1
\mapsisoto\Wedge_H^1\otimes\bigodot^2\!\Wedge_H^1$ for the 
isomorphism~(\ref{algebra}), then 
$$\hat\nabla_Hg=\nabla_Hg-2\sigma\Gamma_H$$
and so $\hat\nabla_Hg=0$ if and only if 
$\Gamma_H=\frac12\sigma^{-1}\nabla_Hg$, which shows both existence and 
uniqueness from our final requirement.
\end{proof}

\section{Other constructions}\label{others}
As noted in~\cite{fgv} and echoed in Proposition~\ref{3D}, sub-Riemannian
geometry in dimension 3 coincides with Webster's pseudo-Hermitian
geometry~\cite{w}. For completeness, we briefly recount the story in 
higher dimensions as follows.

\begin{defn} Suppose $M$ is a smooth manifold of dimension $2n+1$. An 
{\em almost CR structure\/} on $M$ is a vector sub-bundle $H\subset TM$ of rank
$2n$ with an endomorphism $J:H\to H$ such that $J^2=-{\mathrm{Id}}$.
\end{defn}
\begin{defn}
An almost CR structure is said to be {\em non-degenerate\/} if $H$ is a contact
distribution. 
\end{defn}
\begin{defn} An almost CR structure is said to be {\em partially integrable\/} 
if and only if the $L^*$-valued form 
${\mathcal{L}}(X,JY)$ on $H$ is symmetric. 
Equivalently, for any contact form, the ${\mathbb{C}}$-valued form 
$$\Omega(X,JY)-i\Omega(X,Y)$$
on $H$ is Hermitian (and, in this case, non-degeneracy of the CR structure is 
equivalent to non-degeneracy of this Hermitian form).  
\end{defn} 
It is observed in~\cite[p.~414]{cs} that partial integrability is implied by 
the more usual condition of integrability, which may be defined as follows.
\begin{defn}
An almost CR structure is said to be {\em integrable\/} if and only if
$[H^{0,1},H^{0,1}]\subseteq H^{0,1}$ where $H^{0,1}=\{X\in{\mathbb{C}}H\mbox{
s.t.\ }JX+iX=0\}$. Evidently, this condition is vacuous in three dimensions
(for then $H^{0,1}$ is a line bundle). A {\em CR structure\/} is an integrable
almost CR structure.
\end{defn}
\begin{defn}
A {\em pseudo-Hermitian\/} structure is a CR structure equipped with a choice 
of contact form. Such a structure is said to be {\em strictly pseudo-convex\/} 
if and only if the corresponding symmetric form $\Omega(X,JY)$ on $H$ is 
positive-definite.
\end{defn}
\begin{prop} Always, 
$$\begin{array}{l}
\{\mbox{strictly pseudo-convex pseudo-Hermitian structures}\}\\
\hspace{100pt}\subseteq\{\mbox{oriented sub-Riemannian contact structures}\}
\end{array}$$ 
with equality in $3$ dimensions.
\end{prop}
\begin{proof} Using $\Omega(X,JY)$ as a sub-Riemannian metric and using $J$ to 
orient~$H$, it is clear in the co-frames (\ref{lovely_co-frame}) that CR 
geometry corresponds exactly to the case $\lambda_1=\cdots=\lambda_n=1$.
\end{proof}

\subsection{The Tanaka-Webster connection} Since it is only in $3$ dimensions
that pseudo-Hermitian geometry coincides with sub-Riemannian geometry, we shall
confine our discussion to this case. The construction~\cite{t,w} of this
canonical connection, written from the sub-Riemannian point of view, is as
follows. Choose, a local co-frame $\theta,e_1,e_2$ on $M$ such that 
\begin{equation}\label{choose_co-frame}
d\theta=e_1\wedge e_2\quad\mbox{and}\quad Je_2=e_1.\end{equation}
Notice that such a co-frame is determined up to
\begin{equation}\label{change_of_frame}
\left[\begin{array}{c}e_1\\ e_2\end{array}\right]\longmapsto
\left[\begin{array}{c}\hat e_1\\ \hat e_2\end{array}\right]=
\left[\begin{array}{cc}\cos\phi&-\sin\phi\\ \sin\phi&\cos\phi\end{array}\right]
\left[\begin{array}{c}e_1\\ e_2\end{array}\right]\end{equation}
for an arbitrary smooth function~$\phi$.
\begin{lemma} There is a smooth $1$-form $\omega$ and smooth functions $A$ and 
$B$ uniquely characterised by 
\begin{equation}\label{structure_equations}
\begin{array}{rcl}de_1&=&\omega\wedge e_2+A\theta\wedge e_1+B\theta\wedge e_2\\
de_2 &=&-\omega\wedge e_1+B\theta\wedge e_1-A\theta\wedge e_2
\end{array}\end{equation}
\end{lemma}
\begin{proof} At each point, there are seemingly $6$ equations here for $5$
unknowns, namely the $3$ coefficients of $\omega$ together with $A$ and~$B$. 
However, there is one relation namely
$$0=d^2\theta=d(e_1\wedge e_2)=de_1\wedge e_2-de_2\wedge e_1,$$
which is exactly as required by the right hand side 
of~(\ref{structure_equations}).
\end{proof}
Notice that if we change our co-frame according to (\ref{change_of_frame}),
then
$$\begin{array}{rcl}d\hat e_1&=&\hat\omega\wedge\hat e_2+
\hat A\theta\wedge\hat e_1+\hat B\theta\wedge\hat e_2\\
d\hat e_2 &=&-\hat\omega\wedge\hat e_1+
\hat B\theta\wedge\hat e_1-\hat A\theta\wedge\hat e_2,
\end{array}$$
where
\begin{equation}\label{resulting_changes}
\hat\omega=\omega-d\phi\enskip\mbox{ and}\enskip
\left[\begin{array}{c}\hat A\\ \hat B\end{array}\right]=
\left[\begin{array}{cc}\cos2\phi&-\sin2\phi\\ 
\sin2\phi&\cos2\phi\end{array}\right]
\left[\begin{array}{c}A\\ B\end{array}\right].\end{equation}
\begin{thm}[Tanaka-Webster] The connection on $\Wedge^1$ given by
\begin{equation}\label{tanaka-webster_in_our_co-frame}
\nabla\theta=0,\qquad\nabla e_1=\omega\otimes e_2,\qquad
\nabla e_2=-\omega\otimes e_1\end{equation}
in any chosen co-frame, does not depend on this choice.
\end{thm}
\begin{proof}There is no choice in~$\theta$. Otherwise, the required invariance 
follows by straightforward computation from $\hat\omega=\omega-d\phi$.
\end{proof}
We shall now use the co-frame $\theta,e_1,e_2$ and its structure equations 
(\ref{structure_equations}) to
\begin{itemize}
\item compute the torsion of the Tanaka-Webster connection,
\item compute the curvature of the Tanaka-Webster connection,
\item compute the partial connection of Theorem~\ref{mainthm},
\item promote it to a full connection via Proposition~\ref{promotion},
\item and compare these two connections.
\end{itemize}
\subsubsection{Tanaka-Webster torsion}
The torsion of any connection on $\Wedge^1$ is the difference 
between $d:\Wedge^1\to\Wedge^2$ and the composition 
$\Wedge^1\xrightarrow{\,\nabla\,}\Wedge^1\otimes\Wedge^1\to\Wedge^2$.
According to (\ref{structure_equations}), for the Tanaka-Webster connection, 
this is
$$\theta\mapsto e_1\wedge e_2,\qquad
e_1\mapsto A\theta\wedge e_1+B\theta\wedge e_2,\qquad
e_2\mapsto B\theta\wedge e_1-A\theta\wedge e_2,$$
the first of which is just the Levi form and the rest may be written as
\begin{equation}\label{tanaka-webster_torsion}
\left[\begin{array}{c}e_1\\ e_2\end{array}\right]\longmapsto
\theta\wedge\left[\begin{array}{cc}A&B\\ B&\!-A\!\end{array}\right]
\left[\begin{array}{c}e_1\\ e_2\end{array}\right].\end{equation}
Its invariance under change of co-frame (\ref{change_of_frame}) is equivalent 
to the second part of~(\ref{resulting_changes}), which, for 
these purposes may be better rewritten as
$$\left[\begin{array}{cc}\hat A&\hat B\\
\hat B&\!-\hat A\!\end{array}\right]
\left[\begin{array}{cc}\cos\phi&-\sin\phi\\ 
\sin\phi&\cos\phi\end{array}\right]=
\left[\begin{array}{cc}\cos\phi&-\sin\phi\\ 
\sin\phi&\cos\phi\end{array}\right]
\left[\begin{array}{cc}A&B\\ B&\!-A\!\end{array}\right].$$

In the standard expositions, the torsion is usually presented as a
complex-valued quantity, equivalent to $A+iB$.

\subsubsection{Tanaka-Webster curvature}
The curvature of a general connection $E\to\Wedge^1\otimes E$ is the 
composition $E\xrightarrow{\,\nabla\,}\Wedge^1\otimes E
\xrightarrow{\,\nabla\,}\Wedge^2\otimes E$ where
$$\nabla(\alpha\otimes\sigma)=d\alpha\otimes\sigma-\alpha\wedge\nabla\sigma
\quad\mbox{characterises}\quad\nabla:\Wedge^1\otimes E\to\Wedge^2\otimes E.$$
Therefore, we may compute, according to (\ref{tanaka-webster_in_our_co-frame}),
that
$$\begin{array}{rcccccl}
\theta&\xrightarrow{\,\nabla\,}&0\\
e_1&\xrightarrow{\,\nabla\,}&\omega\otimes e_2&\xrightarrow{\,\nabla\,}&
d\omega\otimes e_2-\omega\wedge(-\omega\otimes e_1)&=&d\omega\otimes e_2\\
e_2&\xrightarrow{\,\nabla\,}&-\omega\otimes e_1&\xrightarrow{\,\nabla\,}&
-d\omega\otimes e_1+\omega\wedge(\omega\otimes e_2)&=&-d\omega\otimes e_1
\end{array}$$
In other words, the curvature is determined by $d\omega$. Its invariance is
clear from the first equation of~(\ref{resulting_changes}). In fact, the
curvature provides only one new scalar quantity, namely $d\omega\wedge\theta$,
since $d\omega\wedge e_1$ and $d\omega\wedge e_2$ may be determined in terms of
the torsion by differentiating the structure
equations~(\ref{structure_equations}). It is traditionally captured by the
real-valued function $R$ determined by
\begin{equation}\label{tanaka-webster_curvature}
d\omega\wedge\theta=R\,\theta\wedge e_1\wedge e_2.\end{equation}

\subsection{The partial connection} 
The same co-frame (\ref{choose_co-frame}) may also be used to compute the
partial connection. The characterising properties (\ref{reeb}) of the Reeb
field $T$ show that it is, equivalently, determined by 
$$T\intprod\theta=1\qquad T\intprod e_1=0\qquad T\intprod e_2=0$$
whence the co-frame $\{\theta,e_1,e_2\}$ is compatible with the splitting
(\ref{splitting}). More specifically $\{e_1,e_2\}$ spans
$\Wedge_H^1\hookrightarrow\Lambda^1$ and $\theta$ trivialises
$L\subset\Wedge^1$. Therefore, if we consider the partial connection $\nabla_H$
on $\Wedge^1$ defined by
\begin{equation}\label{our_partial_connection_in_three_dimensions}
\textstyle \nabla_H\theta=e_1\otimes e_2-e_2\otimes e_1\quad
\nabla_He_1=\omega_H\otimes e_2\quad \nabla_He_2=-\omega_H\otimes e_1,
\end{equation}
where $\omega$ is defined by (\ref{structure_equations}) and $\omega_H$ is its 
image in $\Wedge_H^1$, then we ensure that
it has the form (\ref{we_insist}) and is free from partial torsion, as required
by Theorem~\ref{mainthm}. Finally, 
$$\begin{array}{l}
\nabla_H(e_1\otimes e_1+e_2\otimes e_2)\\
\qquad{}=\omega_H\otimes e_2\otimes e_1+\omega_H\otimes e_1\otimes e_2
-\omega_H\otimes e_1\otimes e_2-\omega_H\otimes e_2\otimes e_1=0\end{array}$$
and all characterising properties of Theorem~\ref{mainthm} are satisfied. Thus,
apart from a minor modification whereby $\nabla_H\theta=d\theta$ replaces
$\nabla\theta=0$, the partial connection of Theorem~\ref{mainthm} is induced by
the Tanaka-Webster connection.

\subsubsection{Promotion of the partial connection}
We shall now take the partial connection defined by
(\ref{our_partial_connection_in_three_dimensions}) and promote it to a full
connection on $\Wedge^1$ in line with Proposition~\ref{promotion}. The general
lift of (\ref{our_partial_connection_in_three_dimensions}) to a full connection
is defined by
$$\begin{array}{ccl}
\nabla\theta&=&\theta\otimes\alpha+e_1\otimes e_2-e_2\otimes e_1\\
\nabla e_1&=&\theta\otimes\beta+\omega\otimes e_2\\
\nabla e_2&=&\theta\otimes\gamma-\omega\otimes e_1
\end{array}$$
for $1$-forms $\alpha,\beta,\gamma$ and if we now compute the composition
$$\Wedge^1\xrightarrow{\,\nabla\,}\Wedge^1\otimes\Wedge^1
\xrightarrow{\,\nabla\,}\Wedge^2\otimes\Wedge^1\longrightarrow
\Wedge_H^2\otimes\Wedge^1$$
for this lift, we find that
$$\textstyle\theta\mapsto
e_1\wedge e_2\otimes\alpha
+(de_1+e_2\wedge\omega)_H\otimes e_2
-(de_2-e_1\wedge\omega)_H\otimes e_1
=e_1\wedge e_2\otimes\alpha$$
in accordance with~(\ref{structure_equations}), and then
$$\begin{array}{rcl}
e_1&\mapsto&
e_1\wedge e_2\otimes\beta+(d\omega)_H\otimes e_2\\
e_2&\mapsto&
e_1\wedge e_2\otimes\gamma-(d\omega)_H\otimes e_1.
\end{array}$$
Therefore, we are obliged to take $\alpha=0$ and
$$\beta=-Re_2\quad\mbox{and}\quad\gamma=Re_1,$$
where $(d\omega)_H=R\,e_1\wedge e_2$. In summary, our promoted connection is 
given by 
\begin{equation}\label{our_promoted_connection}
\begin{array}{ccl}
\nabla\theta&=&e_1\otimes e_2-e_2\otimes e_1\\
\nabla e_1&=&(\omega-R\,\theta)\otimes e_2\\
\nabla e_2&=&(R\,\theta-\omega)\otimes e_1
\end{array}\end{equation}
where $R$ is the Tanaka-Webster curvature determined 
by~(\ref{tanaka-webster_curvature}).

\subsection{Comparison}
We may compare the promoted connection~(\ref{our_promoted_connection}) with
Tanaka-Webster. {From} (\ref{tanaka-webster_in_our_co-frame}) and 
(\ref{choose_co-frame}) we find
that their difference tensor, as a homomorphism
$\Wedge^1\to\Wedge^1\otimes\Wedge^1$, is given by
$$\Wedge^1=\begin{array}{c}\Lambda_H^1\\ \oplus\\ \Wedge^0\end{array}
\ni\left[\begin{array}c\sigma\\ \rho\end{array}\right]\mapsto
\left[\begin{array}c
R\,\theta\otimes J\sigma+\Omega\rho\\ 0\end{array}\right],$$
where $R$ is the Webster-Tanaka curvature~(\ref{tanaka-webster_curvature}) and 
the $1$-forms are split by the Reeb field corresponding 
to~$\theta$. 

\begin{rmk}
Recall that the two basic invariants of pseudo-Hermitian geometry are the
torsion (\ref{tanaka-webster_torsion}) and
curvature~(\ref{tanaka-webster_curvature}). Finally, we remark that if we
compute the full torsion of our promoted
connection~(\ref{our_promoted_connection}), then we find $\theta\mapsto 0$ and
$$\left[\begin{array}{c}e_1\\ e_2\end{array}\right]\longmapsto
\theta\wedge\left[\begin{array}{cc}A&B+R\\ B-R&\!-A\!\end{array}\right]
\left[\begin{array}{c}e_1\\ e_2\end{array}\right].$$
In this formula we see the basic invariants appearing together.\end{rmk}

\begin{rmk}
For any strictly pseudo-convex pseudo-Hermitian structure (in any dimension)
there is, apart from the Tanaka--Webster connection, yet another canonical
affine connection, namely the associated Weyl connection defined as in
\cite{cs}. As partial connections on $\Wedge^1$ they coincide, but as full
connections they differ as computed in \cite[Theorem 5.2.13]{cs}. In dimension
$3$ their difference tensor is simply a constant multiple of
$$\left[\begin{array}c\sigma\\ \rho\end{array}\right]\mapsto
\left[\begin{array}c
R\,\theta\otimes J\sigma\\ 0\end{array}\right].$$
\end{rmk}

\end{document}